\documentclass[a4paper,12pt]{amsart}

\usepackage{color}
\usepackage[english]{babel}
\usepackage{amsmath,amssymb,amsfonts,amsthm,enumerate,dsfont}
\usepackage{mathrsfs,epsfig,graphicx}
\usepackage{bbm}

\numberwithin{equation}{section} 

\def\trait#1#2#3{\vrule width #1pt height #2pt depth #3pt}

\def\trait#1#2#3{\vrule width #1pt height #2pt depth #3pt}
\def\trace{\hskip-1pt\mathrel{\hbox{\trait{.3}{7}{0}%
\trait{6}{.3}{0}}}}
\newtheorem{theorem}{Theorem}[section]
\newtheorem{lemma}[theorem]{Lemma}
\newtheorem{definition}[theorem]{Definition}
\newtheorem{remark}[theorem]{Remark}
\newtheorem{proposition}[theorem]{Proposition}

\newtheorem{example}[theorem]{Example}

\newenvironment{Proof}{\removelastskip\par\medskip 
\noindent{\em Proof.}
\rm}{\penalty-20\null\hfill$\square$\par\medbreak}

\newcommand{\Q}{\mathbb{Q}}
\newcommand{\N}{\mathbb{N}}
\newcommand{\R}{\mathbb{R}}

\newcommand{\Z}{\mathbb{Z}}

\newcommand{\Haus}[1]{{\mathscr H}^{#1}} 

\title{The Steiner tree problem revisited through rectifiable $G$-currents}
\author{Andrea Marchese, Annalisa Massaccesi}

\begin{document}

\begin{abstract}
The Steiner tree problem can be stated in terms of finding a connected set of minimal length containing a given set of finitely many points. We show how to formulate it as a mass-minimization problem for $1$-dimensional currents with coefficients in a suitable normed group. The representation used for these currents allows to state a calibration principle for this problem. We also exhibit calibrations in some examples.
\end{abstract}
\maketitle

\section*{Introduction}
The classical {\em Steiner tree problem}\index{Steiner tree problem} consists in finding the shortest connected set containing $n$ given distinct points $p_1,\ldots,p_n$ in $\R^d$. Some very well-known examples are shown in Figure \ref{fig:sol_steiner}.

\begin{figure}[htbp]
\begin{center}
\scalebox{1}{
\input{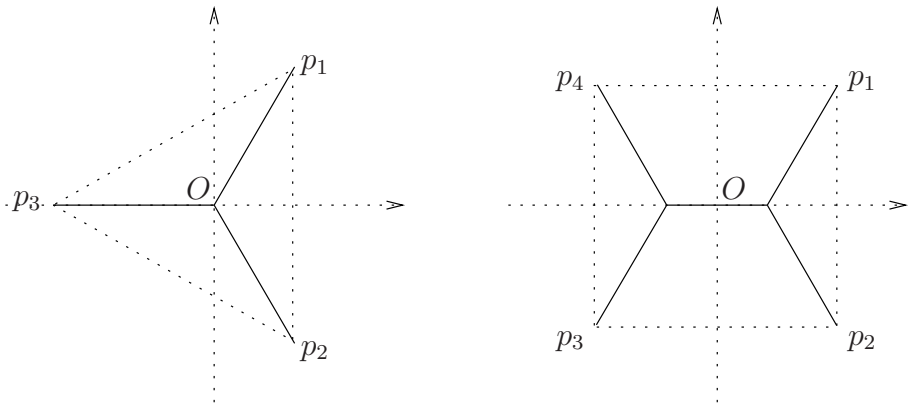}
}
\end{center}
\caption{Solutions for the vertices of an equilateral triangle and a square}
\label{fig:sol_steiner}
\end{figure}

The problem is completely solved in $\R^2$ and there exists a wide literature on the subject, mainly devoted to improving the efficiency of algorithms for the construction of solutions: see, for instance, \cite{gilbert_pollak} and \cite{courant_robbins} for a survey of the problem. The recent papers \cite{stepanov} and \cite{ulivi} witness the current studies on the problem and its generalizations.

Our aim is to rephrase the Steiner tree problem as an equivalent mass minimization problem by replacing connected sets with 1-currents with coefficients in a more suitable group than $\Z$, in such a way that solutions of one problem correspond to solutions of the other, and vice-versa. The use of currents allows to exploit techniques and tools from the Calculus of Variations and the Geometric Measure Theory.



Let us briefly point out a few facts suggesting that classical polyhedral chains with integer coefficients might not be the correct environment for our problem. First of all, one should make the given points $p_1,\ldots,p_n$ in the Steiner problem correspond to some integral polyhedral $0$-chain supported on $p_1,\ldots,p_n$, with suitable multiplicities $m_1,\ldots,m_n$. One has to impose that $m_1+\ldots+m_n=0$ in order that this $0$-chain is the boundary of a compactly supported $1$-chain. In the example of the equilateral triangle, see Figure \ref{fig:sol_steiner}, the condition $m_3=-(m_1+m_2)$ forces to break symmetry, leading to the minimizer in Figure \ref{fig:sol_cl_curr_square}. The desired solution is instead depicted in Figure \ref{fig:sol_steiner}. In the second example from Figure \ref{fig:sol_steiner}, we get the ``wrong'' non-connected minimizer even though all boundary multiplicities have modulus $1$; see Figure \ref{fig:sol_cl_curr_square}.

\begin{figure}[htbp]
\begin{center}
\scalebox{1}{
\input{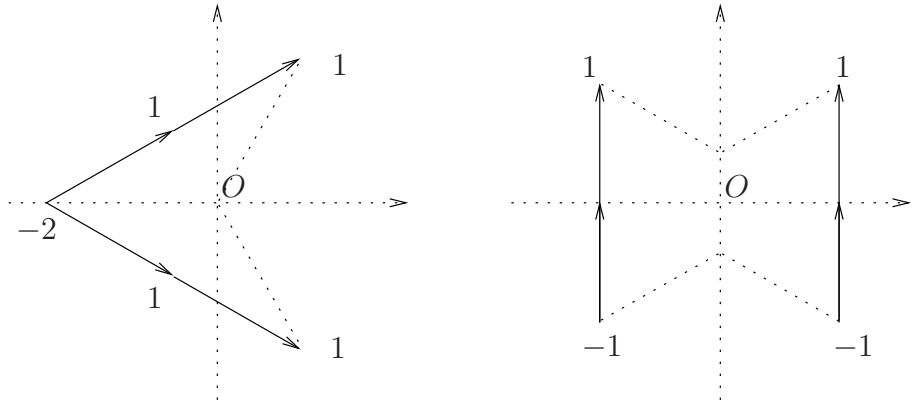}
}
\end{center}
\caption{Solutions for the mass minimization problems among polyhedral chains with integer coefficients}
\label{fig:sol_cl_curr_square}
\end{figure}

\noindent These examples show that $\Z$ is not the right group of coefficients.


Our framework will be that of currents with coefficients in a normed abelian group $G$ (briefly: $G$-currents), which we will introduce in \S \ref{sec:beginning}.


Currents with coefficients in a group were introduced by W. Fleming. There is a vast literature on the subject: let us mention only the seminal paper \cite{Fleming}, the work of B. White \cite{white1,white2}, and the more recent papers by T. De Pauw and R. Hardt \cite{depauw} and by L. Ambrosio and M. G. Katz \cite{ambrosio_2}. A Closure Theorem holds for these flat $G$-chains, see \cite{Fleming} and \cite{white2}.


In \S \ref{sec:steiner} we recast the Steiner problem in terms of a mass minimization problem over currents with coefficients in a discrete group $G$, chosen only on the basis of the number of boundary points. As we already said, this construction provides a way to pass from a mass minimizer to a Steiner solution and vice-versa.


This new formulation permits to initiate a study of calibrations as a sufficient condition for minimality; this is the subject of \S \ref{sec:calibrations}. Classically a calibration\index{calibration} $\omega$ associated with a given oriented $k$-submanifold $S\subset\R^d$ is a unit closed $k$-form taking value $1$ on the tangent space of $S$.
The existence of a calibration guarantees the minimality of $S$ among oriented submanifolds with the same boundary $\partial S$. Indeed, Stokes Theorem and the assumptions on $\omega$ imply that
$$
{\rm vol}(S)=\int_S\omega=\int_{S'}\omega\le{\rm vol}(S'),
$$
for any submanifold $S'$ having the same boundary of $S$.

In order to define calibrations in the framework of $G$-currents, it is convenient to view currents as linear functionals on forms, which is not always possible in the usual setting of currents with coefficients in groups. This motivates the preliminary work in \S \ref{sec:beginning}, where we embed the group $G$ in a normed linear space $E$ and we construct the currents with coefficients in $E$ in the classical way. In Definition \ref{def:gen_calib}, the notion of calibration is slightly weakened in order to include piecewise smooth forms, which appear in Examples \ref{ex:calib_square} and \ref{ex:calib_hex_vittone}, where we exhibit calibrations for the problem on the right of Figure \ref{fig:sol_steiner} and for the Steiner tree problem on the vertices of a regular hexagon plus the center. It is worthwhile to note that our theory works for the Steiner tree problem in $\R^d$ and for currents supported in $\R^d$; we made explicit computations only on $2$-dimensional configurations for simplicity reasons. We conclude \S \ref{sec:calibrations} with some remarks concerning the use of calibrations in similar contexts, see for instance \cite{morgan}.

The existence of a calibration is a sufficient condition for a manifold to be a mi\-ni\-mi\-zer; one could wonder whether this condition is  necessary as well. In general, a smooth (or piecewise smooth, according to Definition \ref{def_compatib_calibration}) calibration might not exist; nevertheless, one can still search for some {\em weak} calibration, for instance a differential form with bounded measurable coefficients. In \S \ref{sec:open_pb} we discuss a strategy in order to get the existence of such a weak calibration. A duality argument due to H. Federer  \cite{federer} ensures that a weak calibration exists for mass-minimizing normal currents; the same argument works for mass-minimizing normal currents with coefficients in the normed vector space $E$. Therefore an equivalence principle between minima among normal and rectifiable $1$-currents with coefficients in $E$ and $G$, respectively, is sufficient to conclude that a calibration exists. Proposition \ref{thm_minmin_classic} guarantees that the equivalence between minima holds in the case of classical $1$-currents with real coefficients; hence a weak calibration always exists. The proof of this result is subject to the validity of a homogeneity property for the candidate minimizer stated in Remark \ref{rmk:homog_cond}. Example \ref{ex_casetta_piccola} shows that for $1$-dimensional $G$-currents an interesting new phenomenon occurs, since (at least in a non-Euclidean setting) this homogeneity property might not hold; the validity of the homogeneity property may be related to the ambient space. The problem of the existence of a calibration in the Euclidean space is still open.\medskip

{\em Acknowledgements.} The authors warmly thank Professor Giovanni Alberti for having posed the problem and for many useful discussions.


\section{Rectifiable currents over a coefficient group}\label{sec:beginning}

In this section we provide definitions for currents over a coefficient group, with some basic examples.

Fix an open set $U\subset\R^d$ and a normed vector space $(E,\|\cdot\|_E)$ with finite dimension $m\ge 1$. We will denote by $(E^*,\|\cdot\|_{E^*})$ its dual space endowed with the dual norm
$$
\|f\|_{E^*}:=\sup_{\|v\|_E\le 1}\langle f;v\rangle\ .
$$


\begin{definition}\label{def_form}
{\rm
We say that a map $$\omega:\Lambda_k(\R^d)\times E\to\R$$ is an $E^*$-valued $k$-covector\index{$E^*$-valued $k$-covector} in $\R^d$ if
\begin{itemize}
\item[(i)]$\forall\,\tau\in \Lambda_k(\R^d),\quad\omega(\tau,\cdot)\in E^*$, that is $\omega(\tau,\cdot):E\to\R$ is a linear function.
\item[(ii)]$\forall\, v\in E,\quad\omega(\cdot,v):\Lambda_k(\R^d)\to\R$ is a (classical) $k$-covector.
\end{itemize}

\noindent Sometimes we will use $\langle\omega;\tau,v\rangle$\index{$\langle\omega;\tau,v\rangle$} instead of $\omega(\tau,v)$, in order to simplify the notation. The space of $E^*$-valued $k$-covectors in
$\R^d$ is denoted by $\Lambda^k_E(\R^d)$\index{$\Lambda^k_E(\R^d)$} and it is endowed with the comass norm
\begin{equation}\label{covector_norm}
\|\omega\|:=\sup\left\{\left\|\omega(\tau,\cdot)\right\|_{E^*}\,:\,|\tau|\leq 1,\tau\ {\rm simple}\right\}\ .
\end{equation}
}
\end{definition}


\begin{remark}\label{rmk:components}
{\rm
Fix an orthonormal system of coordinates in $\R^d$, $({\bf e}_1,\ldots,{\bf e}_d)$; the corresponding dual base in $(\R^d)^{*}$ is $(dx_1,\ldots,dx_d)$. Consider a complete biorthonormal system for $E$, i.e., a pair
$$(v_1,\dots,v_m)\in E^m;\ (w_1,\ldots,w_m)\subset (E^*)^m$$ such that $\|v_i\|_E=1$, $\|w_i\|_{E^*}=1$ and $\langle w_i; v_j \rangle=\delta_{ij}$. Given an $E^*$-valued $k$-covector $\omega$, we denote
$$
\omega^j:=\omega(\cdot,v_j).\index{$\omega^j$}
$$
For each $j\in\{1,\ldots, m\}$, $\omega^j$ is a $k$-covector in the usual sense.
Hence the bi\-or\-tho\-nor\-mal system $(v_1,\dots,v_m)$, $(w_1,\ldots,w_m)$ allows to write $\omega$ in ``components'' $$
\omega=(\omega^1,\ldots,\omega^m)\,,$$
in fact we have
$$
\omega(\tau,v)=\sum_{j=1}^m \langle\omega^j; \tau\rangle\langle w_j;v\rangle\ .
$$
In particular $\omega^j$ admits the usual representation
$$\omega^j=\sum_{1\leq i_1<\ldots<i_k\leq d} a^j_{i_1\ldots i_k}dx_{i_1}\wedge\ldots\wedge dx_{i_k}\,,\qquad j=1,\ldots,m.$$
}
\end{remark}


\begin{definition}\label{def_form2}
{\rm
An $E^*$-valued differential $k$-form\index{$E^*$-valued differential $k$-form} in $U\subset\R^d$, or just a $k$-form when it is clear which vector space we are referring to, is a map
$$
\omega:U\to\Lambda^k_E(\R^d) ;
$$
we say that $\omega$ is $\mathscr{C}^\infty$-regular if every component $\omega^j$ is so (see Remark \ref{rmk:components}).
We denote by $\mathscr{C}_c^\infty(U,\Lambda^k_E(\R^d))$\index{$\mathscr{C}_c^\infty(U,\Lambda^k_E(\R^d))$} the vector space of $\mathscr{C}^\infty$-regular $E^*$-valued $k$-forms with compact support in $U$.
}
\end{definition}

We are mainly interested in $E^*$-valued $1$-forms, nevertheless we analyze $k$-forms in wider generality, in order to ease other definitions, such as the differential of an $E^*$-valued form and the boundary of
an $E$-current.

\begin{definition}{\rm
We define the differential ${\rm{d}}\omega$ of a $\mathscr{C}^\infty$-regular $E^*$-valued $k$-form $\omega$ by components:
$$
{\rm{d}}\omega^j:={\rm{d}}(\omega^j):U\to\Lambda^{k+1}(\R^d)\,,\qquad j=1,\ldots,m\ ,
$$
Moreover, $\mathscr{C}^\infty_c(U,\Lambda^1_E(\R^d))$ has a norm, denoted by $\|\cdot\|$, given by the supremum of the comass norm of the form defined in \eqref{covector_norm}. Hence we mean
\begin{equation}\label{norm_kform}
\|\omega\|:=\sup_{x\in U}\|\omega(x)\|\ .
\end{equation}
}
\end{definition}


\begin{definition}\label{def_current}
{\rm
A $k$-dimensional current $T$ in $U\subset\R^d$, with coefficients in $E$, or just an $E$-current\index{$E$-current} when there is no doubt on the dimension, is a linear and continuous function
$$
T:\mathscr{C}^\infty_c(U,\Lambda^k_E(\R^d))\longrightarrow\R\ ,
$$
where the continuity is meant with respect to the locally convex topology on the space $\mathscr{C}^\infty_c(U,\Lambda^k_E(\R^d))$, built in analogy with the topology on $\mathscr{C}^{\infty}_c(\R^n)$, with respect
to which distributions are dual. This defines the weak$^*$ topology on the space of $k$-dimensional $E$-currents. Convergence in this topology is equivalent to the convergence of all the ``components'' in the space
of classical\footnote{In the sequel we will use ``classical'' to refer to the usual currents, with coefficients in $\R$ or possibly in $\Z$.} $k$-currents, by which we mean the following. We define for every
$k$-dimensional $E$-current $T$ its components $T^j$, for $j=1,\ldots m$, and we write $$T=(T^1,\ldots,T^m),$$
denoting
$$
\langle T^j;\varphi\rangle:=\langle T;\widetilde\varphi_j\rangle\ ,
$$
for every (classical) compactly supported differential $k$-form $\varphi$ on $\R^d$. Here $\widetilde\varphi_j$ denotes the $E^*$-valued differential $k$-form on $\R^d$ such that
\begin{align}
& \widetilde\varphi_j(\cdot,v_j)=\varphi,\\
& \widetilde\varphi_j(\cdot,v_i)=0\quad{\rm for}\ i\neq j\ .
\end{align}
It turns out that a sequence of $k$-dimensional $E$-currents $T_h$ weakly$^*$ converges to an $E$-current $T$ (in this case we write $T_h\stackrel{\ast}{\rightharpoonup}T$) if and only if the sequence of the
components $T_h^j$ converge to $T^j$ in the space of classical $k$-currents, for $j=1,\ldots,m$.

}
\end{definition}

\begin{definition}\label{def_boundary}
{\rm
For a $k$-current $T$ over $E$ we define the boundary operator
$$
\langle\partial T;\varphi\rangle:=\langle T;{\rm{d}}\varphi\rangle\quad\forall\,\varphi=(\varphi^1,\ldots,\varphi^m)\in \mathscr{C}^\infty_c(U,\Lambda^{k-1}_E(\R^d))
$$
and the mass
$$
\mathds{M}(T):=\sup_{\|\omega\|\leq 1} \langle T;\omega\rangle.
$$

As one can expect, the boundary $\partial (T^j)$ of every component $T^j$ is the relative component $(\partial T)^j$of the boundary $\partial T$.
}
\end{definition}


\begin{definition}\label{def_norm_current}{\rm
A $k$-dimensional normal $E$-current\index{$k$-dimensional normal $E$-current} in $U\subset\R^d$ is an $E$-current $T$ with $\mathds{M}(T)<+\infty$ and $\mathds{M}(\partial T)<+\infty$.
Thanks to the Riesz Theorem, $T$ admits the following representation:
$$
\langle T;\omega\rangle=\int_U\langle\omega(x);\tau(x),v(x)\rangle\,{\rm{d}}\mu_T(x)\ ,\quad\forall\,\omega\in \mathscr{C}^\infty_c(U,\Lambda^{k}_E(\R^d))\ .
$$
where $\mu_T$ is a Radon measure on $U$, $v:U\to E$ is summable with respect to $\mu_T$ and $|\tau|=1$, $\mu_T$-a.e. A similar representation holds for the boundary $\partial T$.

}
\end{definition}

\begin{definition}\label{def_rect_current}
{\rm
A rectifiable $k$-current $T$ in $U\subset\R^d$, over $E$, or a rectifiable $E$-current\index{rectifiable $E$-current} is an $E$-current admitting the following representation:
$$
\langle T;\omega\rangle:=\int_{\Sigma}\langle\omega(x);\tau(x),\theta(x)\rangle\,{\rm{d}}\Haus{k}(x),\quad\forall\,\omega\in \mathscr{C}^\infty_c(\R^d,\Lambda^{k}_E(U))
$$
where $\Sigma$ is a countably $k$-rectifiable set (see Definition 5.4.1 of \cite{krantz}) contained in $U$, $\tau(x)\in T_x\Sigma$ with $|\tau(x)|=1$ for $\Haus{k}$-a.e. $x\in\Sigma$ and
$\theta\in L^1(\Haus{k}\trace\Sigma;E)$. We will refer to such a current as $T=T(\Sigma,\tau,\theta)$.\index{$T(\Sigma,\tau,\theta)$} If $B$ is a Borel set and $T(\Sigma,\tau,\theta)$ is a rectifiable $E$-current,
we denote by $T\trace B$ the current $T(\Sigma\cap B,\tau,\theta)$.
}
\end{definition}

Consider now a discrete subgroup $G<E$, endowed with the restriction of the norm $\|\cdot\|_E$. If the multiplicity $\theta$ takes only values in $G$, and if the same holds in the representation of $\partial T$,
we call $T$ a rectifiable $G$-current. Pay attention to the fact that, in the framework of currents over the coefficient group $E$, rectifiable $E$-currents play the role of (classical) rectifiable current, while
rectifiable $G$-currents correspond to (classical) integral currents. Actually this correspondence is an equality, when $E$ is the group $\R$ (with the Euclidean norm) and $G$ is $\Z$.




The next proposition gives a formula to compute the mass of a 1-dimensional rectifiable $E$-current.

\begin{proposition}\label{prop_massmult}
Let $T=T(\Sigma,\tau,\theta)$ be a 1-dimensional rectifiable $E$-current, then
$$
\mathds{M}(T)=\int_\Sigma\|\theta(x)\|_E\,{\rm{d}}\Haus{1}(x)\ .
$$
\end{proposition}

\noindent Since the mass is lower semicontinuous, we can apply the direct method of
the Calculus of Variations for the existence of minimizers with given
boundary, once we provide the following compactness result. Here we assume for simplicity that $G$ is the subgroup of $E$ generated by $v_1,\ldots,v_m$ (see Remark \ref{rmk:components}). A similar argument works
for every discrete subgroup $G$.

\begin{theorem}\label{thm_rectifiability}
Let $\left(T_h\right)_{h\ge 1}$ be a sequence of rectifiable $G$-currents such that there exists a positive finite constant $C$ satisfying
$$
\mathds{M}(T_h)+\mathds{M}(\partial T_h)\le C\;\;\;{\rm{for\;every\;}}h\ge 1\ .
$$
Then there exists a subsequence $\left(T_{h_i}\right)_{i\ge 1}$ and a rectifiable $G$-current $T$ such that
$$
T_{h_i}\stackrel{\ast}{\rightharpoonup}T.
$$
\end{theorem}

\begin{proof}
The statement of the theorem can be proved component by component. In fact, let $T_h^1,\ldots,T_h^m$ be the components of $T_h$.
Since $(v_1,\ldots,v_m),(w_1,\ldots,w_m)$ is a biorthonormal system, we have
$$
\mathds{M}(T_h^j)+\mathds{M}(\partial T_h^j)\le m(\mathds{M}(T_h)+\mathds{M}(\partial T_h))\le mC\ ,
$$
hence, after a diagonal procedure, we can find a subsequence $\left(T_{h_i}\right)_{i\ge 1}$ such that $\left(T_{h_i}^j\right)_{i\ge 1}$ weakly$^*$ converges to some integral current $T^j$, for every $j=1,\ldots, m$. Denoting by $T$ the rectifiable $G$-current, whose components are $T^1,\ldots,T^m$, we have
$$
T_{h_i}\stackrel{\ast}{\rightharpoonup}T.
$$
\end{proof}

We conclude this section with some notations and basic facts about certain classes of rectifiable $E$-currents. Given a Lipschitz path $\gamma:[0,1]\to\R^2$ (parametrized with constant speed), and a coefficient $g\in G$, we define the associated 1-dimensional rectifiable $G$-current $T=T(\Gamma,\tau,g)$, where $\Gamma$ is the curve $\gamma([0,1])$ and, denoting by $\ell(\Gamma)$ the length of the curve $\Gamma$, the orientation $\tau$
is defined by $\tau(\gamma(t)):=\gamma'(t)/\ell(\Gamma)$ for a.e. $t\in[0,1]$. It turns out that the boundary of such a current is $\partial T=g\delta_{\gamma(1)}-g\delta_{\gamma(0)}$, where the notation means that for every
smooth $E^*$-valued map $\omega$, there holds $\langle \partial T;\omega\rangle=\langle\omega(\gamma(1));g\rangle-\langle\omega(\gamma(0));g\rangle$. Using this notation, we observe that, given some points $P_1,\ldots,P_k$
and some multiplicities $g_1,\ldots,g_k$ in $G$, the $0$-dimensional rectifiable $G$-current
$S=g_1\delta_{P_1}+\ldots+g_k\delta_{P_k}$ is the boundary of some $1$-dimensional rectifiable $G$-current with compact support $T$ if and only if $g_1+\ldots+g_k=0$.


\section{Steiner tree Problem revisited}\label{sec:steiner}

In this section we establish the equivalence between the Steiner tree problem and a mass minimization problem in a family of $G$-currents. We firstly need to choose the right group of coefficients $G$.
Once we fix the number $n$ of points in the Steiner problem, we construct a normed vector space $(E,\|\cdot\|_E)$ and a subgroup $G$ of $E$, satisfying the following properties:
\begin{itemize}
  \item [(P1)] there exist $g_1,\ldots,g_{n-1}\in G$ and $h_1,\ldots,h_{n-1}\in E^*$ such that $(g_1,\ldots,g_{n-1})$ with $(h_1,\ldots,h_{n-1})$ is a complete biorthonormal system for $E$ and $G$ is ge\-ne\-ra\-ted by $g_1,\ldots,g_{n-1}$;
  \item [(P2)] $\|g_{i_1}+\ldots+g_{i_k}\|_E=1\;\;\;{\rm whenever}\ 1\le i_1<\ldots<i_k\le n-1\text{ and }k\le n-1$;
  \item [(P3)] $\|g\|_E\geq 1$ for every $g\in G\setminus\{0\}$;
  \item [(P4)] let $\theta=\sum_{j=1}^{n-1}\theta_jg_j$ and $\widetilde{\theta}=\sum_{j=1}^{n-1}\widetilde{\theta}_jg_j$ satisfy the following condition: $0\le\widetilde{\theta}_j\le \theta_j$ when
 $\theta_j\ge 0$ and $0\ge\widetilde{\theta}_j\ge \theta_j$ otherwise. Then $\|\widetilde{\theta}\|_E\leq\|\theta\|_E$.
\end{itemize}

\noindent For the moment we will assume the existence of $G$ and $E$. The proof of their existence and an explicit representation, useful for the computations, is given in Lemma \ref{representation}.



The next lemma has a fundamental role: through it, we can give a nice structure of 1-dimensional rectifiable $G$-current to every suitable competitor for the Steiner tree problem. From now on we will denote $g_n:=-(g_1+\ldots+g_{n-1})$.

\begin{lemma}\label{struct}
Let $B$ 
be a compact and connected set with finite length in $\R^d$, containing the points $p_1,\ldots,p_n$. Then there exists a connected set $B'\subset B$ containing $p_1,\ldots,p_n$ and a 1-dimensional rectifiable $G$-current
$T_{B'}= T(B',\tau,\theta)$, such that
\begin{itemize}
\item[(i)] $\|\theta(x)\|_{E}=1$ for a.e. $x\in B'$,
\item[(ii)] $\partial T_{B'}$ is the $0$-dimensional $G$-current $g_1\delta_{p_1}+\ldots+g_{n}\delta_{p_{n}}$.
\end{itemize}
\end{lemma}
\begin{Proof}
Since $B$ is a connected, compact set of finite length, then $B$ is connected by paths of finite length (see Lemma 3.12 of \cite{falconer}). Consider a curve $B_1$ which is the image of an injective path
contained in $B$ going from $p_1$ to $p_n$
and associate to it the rectifiable $G$-current $T_1$ with multiplicity $-g_1$, as explained in \S\ref{sec:beginning}. Repeat this procedure keeping the end-point $p_n$ and replacing at each step $p_1$ with
$p_2,\ldots,p_{n-1}$.
To be precise, in this procedure, as soon as a curve $B_i$ intersects an other curve $B_j$ with $j<i$, we force $B_i$ to coincide with $B_j$ from that intersection point to the end-point $p_n$.
The set $B'=B_1\cup\ldots\cup B_{n-1}\subset B$ is a connected set containing $p_1,\ldots,p_n$ and the $1$-dimensional rectifiable $G$-current $T=T_1+\ldots+T_{n-1}$ satisfies the requirements of the lemma,
in particular condition (i) is a consequence of (P2).
\end{Proof}

Via the next lemma (Lemma \ref{connect}), we can say that solutions to the mass minimization problem defined in Theorem \ref{thm:main_steiner} have connected supports. For the proof we need the following theorem
on the structure of classical integral $1$-currents. This theorem has been firstly stated as a corollary of Theorem 4.2.25 in \cite{Fed}. It allows us to consider an integral $1$-current as a countable sum of
oriented simple Lipschitz curves with integer multiplicities.

\begin{theorem}\label{thm:struct_1currents}
Let $T$ be an integral $1$-current in $\R^d$, then
\begin{equation}\label{decomp_1current}
T=\sum_{k=1}^K T_k +\sum_{\ell=1}^{\infty} C_\ell\,,
\end{equation}
with
\begin{enumerate}
\item[(i)] $T_k$ are integral $1$-currents associated to injective Lipschitz paths, for every $k=1,\ldots,K$ and $C_\ell$ are integral $1$-currents associated to Lipschitz paths which have the same value at 0 and 1
and are injective on $(0,1)$, for every $\ell\ge 1$;
\item[(ii)] $\partial C_\ell=0$ for every $\ell\ge 1$.
\end{enumerate}
Moreover
\begin{equation}\label{spezza_massa}
\mathds{M}(T)=\sum_{k=1}^K \mathds{M}(T_k)+\sum_{\ell=1}^{\infty}\mathds{M}(C_\ell)
\end{equation}
and
\begin{equation}\label{spezza_massa_bordo}
\mathds{M}(\partial T)=\sum_{k=1}^K \mathds{M}(\partial T_k)\ .
\end{equation}
\end{theorem}

\begin{lemma}\label{connect}
Let $T=T(\Sigma,\tau,\theta)$ be a $1$-dimensional rectifiable $G$-current such that $\partial T$ is the $0$-current $g_1\delta_{p_1}+\ldots+g_n\delta_{p_n}$. Then there exists a rectifiable $G$-current
$\widetilde T=T(\widetilde{\Sigma},\widetilde{\tau},\widetilde{\theta})$ such that
\begin{itemize}
\item[(i)] $\partial\widetilde T=\partial T=g_1\delta_{p_1}+\ldots+g_n\delta_{p_n}$;
\item[(ii)]${\rm{supp}}(\widetilde{T})$ is a connected $1$-rectifiable set containing $\left\{p_1,\ldots,p_n\right\}$ and it is contained in ${\rm{supp}}(T)$;
\item[(iii)] $\Haus 1({\rm{supp}}(\widetilde{T})\setminus\widetilde\Sigma)=0$;
\item[(iv)] $\mathds{M}(\widetilde T)\le \mathds{M}(T)$ and, if equality holds, then ${\rm{supp}}(T)={\rm{supp}}(\widetilde{T})$.
\end{itemize}
\end{lemma}

\begin{Proof}
Let $T^j=T(\Sigma^j,\tau^j,\theta^j)$ be the components of $T$, for $j=1,\ldots,n-1$ (with  respect to the biorthonormal system $(g_1,\ldots,g_{n-1})$, $(h_1,\ldots,h_{n-1})$).

For every $j$, we can use Theorem \ref{thm:struct_1currents} and write
$$
T^j=\sum_{k=1}^{K_j}T^j_k+\sum_{\ell=1}^{\infty}C^j_\ell\ .
$$
Moreover, since $\partial T^j=\delta_{p_j}-\delta_{p_n}$, by (\ref{spezza_massa_bordo}), we have $K_j=1$ for every $j$. We choose $\widetilde T$ the rectifiable $G$-current whose components are $\widetilde T^j:=T_1^j$.

Because of \eqref{spezza_massa}, we have ${\rm{supp}}(\widetilde T^j)\subset{\rm{supp}}(T^j)$ (the cyclic part of $T^j$ never cancels the acyclic one).

Property (i) is easy to check. Property (iii) is also easy to check, because the corresponding property holds for every component $\widetilde{T}^j$. To prove property (ii), it is sufficient to observe that
$\widetilde T$ is a finite sum of currents associated to oriented curves with multiplicities, having the point $p_n$ in the support and that, by (P1), $g_1,\ldots,g_{n-1}$ are linearly independent, hence the support of $\widetilde T$
is the union of the supports of $\widetilde{T}^j$. The inequality in property (iv) follows from (\ref{spezza_massa}) and from property (P4): indeed (\ref{spezza_massa}) implies that for
every index $\ell$ such that the support of $C^j_{\ell}$ intersects the support of $T^j_1$ in a set of positive length, then $\Haus{1}$-a.e. on this set the orientation of $C^j_{\ell}$ coincide with the
orientation of $T^j_1$. Moreover, if $\mathds{M}(\widetilde T) = \mathds{M}(T)$, then (\ref{spezza_massa}) implies that every cycle $C_{\ell}^j$ is supported in
${\rm{supp}}(\widetilde T)$, hence the second part of (iv) follows.
\end{Proof}

Before stating the main theorem, let us point out that the existence of a solution to the mass minimization problem is a consequence of Theorem \ref{thm_rectifiability}.

\begin{theorem}\label{thm:main_steiner}
Assume that $T_0=T(\Sigma_0,\tau_0,\theta_0)$ is a mass-minimizer among all $1$-dimensional rectifiable $G$-currents with boundary
$$
B=g_1\delta_{p_1}+\ldots+g_n\delta_{p_n}\ .
$$
Then $S_0:={\rm supp}(T_0)$ is a solution of the Steiner tree problem. Conversely, given a set $C$ which is a solution of the Steiner problem for the points $p_1,\ldots,p_n$, there exists a canonical 1-dimensional
$G$-current, supported on $C$, minimizing the mass among the currents with boundary $B$.
\end{theorem}

\begin{Proof}
Since $T_0$ is a mass minimizer, then the mass of $T_0$ must coincide with that of the current $\widetilde{T_0}$ given by the Lemma \ref{connect}. In particular, properties (ii) and (iv) of Lemma \ref{connect}
guarantee that $S_0$ is a connected set.

Let $S$ be a competitor for the Steiner tree problem and let $S'$ and $T_{S'}$ be the connected set and the rectifiable $1$-current given by Lemma \ref{struct}, respectively. Hence we have
\begin{equation*}
\Haus{1}(S)\ge\Haus{1}(S')\stackrel{{\rm{(i)}}}{=}\mathds{M}(T_{S'})\stackrel{{\rm{(ii)}}}{\ge}\mathds{M}(T_0)\stackrel{{\rm{(iii)}}}{\ge}\Haus{1}(\Sigma_0)\stackrel{{\rm{(iv)}}}{=}\Haus{1}(S_0)\ ,
\end{equation*}
indeed
\begin{itemize}
\item[(i)] thanks to the second property of Lemma \ref{struct} and Proposition \ref{prop_massmult}, we obtain
$$
\mathds{M}(T_{S'})=\int_{S'}\|\theta_{S'}(x)\|_E\,{\rm{d}}\Haus{1}(x)=\Haus{1}(S')\ ;
$$
\item[(ii)] we assumed that $T_0$ is a mass-minimizer;
\item[(iii)] from property (P3), we get
$$
\mathds{M}(T_0)=\int_{\Sigma_0}\|\theta_0(x)\|_E\,{\rm{d}}\Haus{1}(x)\ge\int_{\Sigma_0}1\,{\rm{d}}\Haus{1}(x)=\Haus{1}(\Sigma_0)\ ;
$$
\item[(iv)] is property (iii) in Lemma \ref{connect}.
\end{itemize}

To prove the second part of the theorem, apply Lemma \ref{struct} to the set $C$. Notice that with the procedure described in the lemma, the rectifiable $G$-current $T_{C'}$ is uniquely determined,
because for every point $p_i$, $C$ contains exactly one path from $p_i$ to $p_n$, in fact it is well-known that solutions of the Steiner tree problem cannot contain cycles; this explains the adjective
``canonical''. Assume by contradiction there exists a 1-dimensional $G$-current $T$ with $\partial T=B$ and $\mathds{M}(T)<\mathds{M}(T_{C'})$. The 1-dimensional $G$-current $\widetilde T$ obtained applying Lemma
\ref{connect} to $T$ has a connected $1$-rectifiable support containing $\left\{p_1,\ldots,p_n\right\}$ and satisfies
$$\Haus{1}({\rm{supp}}(\widetilde T))\leq\mathds{M}(\widetilde T)\leq \mathds{M}(T)<\mathds{M}(T_{C'})=\Haus{1}({\rm{supp}}(T_{C'})\leq\Haus{1}(C),$$
which is a contradiction.
\end{Proof}
\begin{remark}
\rm{The proof given in the previous theorem shows in particular that the solutions of the mass minimization problem do not depend on the choice of $E$ and $G$, but are universal for every $G$ and $E$
satisfying \rm{(P1)-(P4)}.}
\end{remark}

Eventually, we give an explicit representation for $G$ and $E$.

\begin{lemma}\label{representation}
For every $n\in\N$ there exist a normed vector space $(E,\|\cdot\|_E)$ and a subgroup $G$ of $E$ satisfying {\rm{(P1)-(P4)}}.
\end{lemma}

\begin{Proof}
Let ${\bf e}_1,\ldots,{\bf e}_n$ be the standard basis of $\R^n$ and ${\rm{d}}x_1,\ldots {\rm{d}}x_n$ be the dual basis. Consider
$$
E:=\{v\in\R^n\,:\,v\cdot{\bf e}_n=0\}
$$
and the homomorphism $\phi:\R^n\to E$ such that
\begin{equation}\label{omo_phi}
\phi(u_1,\ldots,u_n):=(u_1-u_n,\ldots,u_{n-1}-u_n,0)\ .
\end{equation}

Consider on $\R^n$ the seminorm
$$
\|u\|_\star:=\max_{i=1,\ldots,n} u\cdot{\bf e}_i-\min_{i=1,\ldots,n} u\cdot{\bf e}_i\ .
$$
and observe that $\|\cdot\|_\star$ induces via $\phi$ a norm on $E$ that we denote $\|\cdot\|_E$.
For every $i=1,\ldots,n-1$, define $g_i:=\phi({\bf e}_i)$ and define $g_n:=-(g_1+\ldots+g_{n-1})$.
Let $G$ be the subgroup of $E$ generated by $g_1,\ldots,g_{n-1}$. For every $i=1,\ldots,n-1$ denote by $h_i$ the element ${\rm{d}}x_i$ of $E^*$.
The pair $(g_1,\ldots,g_{n-1})$, $(h_1,\ldots,h_{n-1})$ is a biorthonormal system and properties (P1)-(P4) are easy to check.
\end{Proof}


\begin{remark}\label{normestar}
\rm{ The norm $\|\cdot\|_{E^*}$ of an element $w=w_1 h_1+\ldots w_{n-1}h_{n-1}\in E^*$ can be characterized in the following way: let us abbreviate $w^P:=\sum_{i=1}^{n-1}(w_i\vee 0)$ and
$w^N:=-\sum_{i=1}^{n-1}(w_i\wedge 0)$ and, for every $v=(v_1,\ldots,v_{n-1},0)\in E$ with $\|v\|_E=1$,  $\lambda(v):=\max_{i=1,\ldots,n-1} (v_i\vee 0)\in[0,1]$, then
\begin{multline}\label{char_norm_Estar}
\|w\|_{E^*}=\sup_{\|v\|_E=1}\sum_{i=1}^{n-1} w_i v_i=\sup_{\|v\|_E=1}[\lambda(v)w^P+(1-\lambda(v))w^N]\\
=\sup_{\lambda\in[0,1]}[(\lambda w^P+(1-\lambda) w^N]=w^P\vee w^N\ .
\end{multline}
Moreover we can also notice that, according to this representation of $E$ and $G$, the only extreme points of the unit ball in $E$ are all the points of $G$ of unit norm, i.e. all the points $g$ of the type
 $g=\pm (g_{i_1}+\ldots+g_{i_k})$ such that $\ 1\le i_1<\ldots<i_k\le n-1\text{ and }k\le n-1$.}
\end{remark}




\section{Calibrations}\label{sec:calibrations}


As we recalled in the Introduction, our interest in calibrations is the reason why we have chosen to provide an integral representation for $E$-currents, indeed the existence of a calibration guarantees the
minimality of the associated current, as we will see in Proposition \ref{prop:calib}.

\begin{definition}\label{def_calib}
{\rm
A smooth calibration associated with a $k$-dimensional rectifiable $G$-current $T(\Sigma,\tau,\theta)$ in $\R^d$ is a smooth compactly supported $E^*$-valued differential $k$-form $\omega$, with the following
properties:
\begin{enumerate}
\item[(i)]$\langle\omega(x);\tau(x),\theta(x)\rangle=\|\theta(x)\|_E$ for $\Haus{k}$-a.e. $x\in\Sigma$;
\item[(ii)] ${\rm{d}}\omega=0$;
\item[(iii)]$\|\omega\|\leq 1$, where $\|\omega\|$ is the comass of $\omega$, defined in (\ref{norm_kform}).
\end{enumerate}
}
\end{definition}
\begin{proposition}\label{prop:calib}
A rectifiable $G$-current $T$ which admits a smooth calibration $\omega$ is a minimizer for the mass among the normal $E$-currents with boundary $\partial T$.

\end{proposition}
\begin{proof}{\rm
Fix a competitor $T'$ which is a normal $E$-current associated with the vectorfield $\tau'$, the multiplicity $\theta'$ and the measure $\mu_{T'}$ (according to Definition \ref{def_norm_current}), with
$\partial T'=\partial T$. Since $\partial(T-T')=0$,
then $T-T'$ is a boundary of some $E$-current $S$ in $\R^d$, and then
\begin{eqnarray}
\mathds{M}(T)&=&\int_\Sigma \|\theta\|_{E}\,{\rm{d}}\Haus{k}\\\label{calib_line}
&\stackrel{\rm(i)}{=}&\int_\Sigma\langle\omega(x);\tau(x),\theta(x)\rangle\,{\rm{d}}\Haus{k}=\langle T;\omega\rangle\\\label{calib_line2}
&\stackrel{\rm(ii)}{=}&\langle T';\omega\rangle=\int_{\R^d}\langle\omega(x);\tau'(x),\theta'(x)\rangle\,{\rm{d}}\mu_{T'}\\\label{calib_line3}
&\stackrel{\rm(iii)}{\leq}&\int_{\R^d}\|\theta'\|_E\,{\rm{d}}\mu_{T'}=\mathds{M}(T')\ ,\label{calib_line4}
\end{eqnarray}
where each equality (respectively inequality) holds because of the corresponding property of $\omega$, as established in Definition \ref{def_calib}. In particular, equality in (ii) follows from
$$\langle T-T';\omega\rangle=\langle\partial S;\omega\rangle=\langle S;{\rm{d}}\omega\rangle=0.$$}
\end{proof}
\begin{remark}{\rm
If $T$ is a rectifiable $G$-current calibrated by $\omega$, then every mass minimizer with boundary $\partial T$ is calibrated by the same form $\omega$. In fact, choose a mass minimizer
$T'=T(\Sigma',\tau',\theta')$ with boundary $\partial T'=\partial T$: obviously we have $\mathds{M}(T)=\mathds{M}(T')$, then equality holds in \eqref{calib_line3}, which means
$$
\langle\omega(x);\tau'(x),\theta'(x)\rangle=\|\theta'(x)\|_E\quad{\rm{for}}\;\Haus{k}-{\rm a.e.}\;x\in\Sigma'\ .
$$
}
\end{remark}

At this point we need a short digression on the representation of a $E^*$-valued $1$-form $\omega$; we will consider the case $d=2$, all our examples being for the Steiner tree problem in $\R^2$. Remember that
in \S \ref{sec:steiner} we fixed a basis $(h_1,\ldots,h_{n-1})$ for $E^*$, dual to the basis $(g_1,\ldots,g_{n-1})$ for $E$. We represent
\begin{equation*}
\omega=\left(
\begin{array}{c}
\omega_{1,1}\,{\rm{d}}x_1+\omega_{1,2}\,{\rm{d}}x_2\\
\vdots\\
\omega_{n-1,1}\,{\rm{d}}x_1+\omega_{n-1,2}\,{\rm{d}}x_2
\end{array}
\right)\ ,
\end{equation*}
so that, if $\tau=\tau_1{\bf e}_1+\tau_2{\bf e}_2\in \Lambda_1(\R^2)$ and $v=v_1g_1+\ldots+v_{n-1}g_{n-1}\in E$, then
$$
\langle\omega;\tau,v\rangle=\sum_{i=1}^{n-1}v_i(\omega_{i,1}\tau_1+\omega_{i,2}\tau_2)\ .
$$

\begin{example}\label{ex:triangle}{\rm
Consider the vector space $E$ and the group $G$ defined in Lemma \ref{representation} with $n=3$; let
$$p_0=(0,0),p_1=(1/2,\sqrt{3}/2),p_2=(1/2,-\sqrt{3}/2),p_3=(-1,0)$$
(see Figure \ref{fig:triang_sol}). Consider the rectifiable $G$-current $T$ supported in the cone over $(p_1,p_2,p_3)$, with respect to $p_0$, with piecewise constant weights $g_1,g_2,g_3=:-(g_1+g_2)$
on $\overline{p_0p_1},\overline{p_0p_2},\overline{p_0p_3}$
respectively (see Figure \ref{fig:triang_sol} for the orientation). This current $T$ is a minimizer for the mass.
In fact, a constant $G$-calibration $\omega$ associated with $T$ is
$$\omega:=\left(\begin{array}{c}
\frac{1}{2}\, {\rm{d}}x_1 + \frac{\sqrt{3}}{2}\, {\rm{d}}x_2\\
\frac{1}{2}\, {\rm{d}}x_1 - \frac{\sqrt{3}}{2}\, {\rm{d}}x_2
\end{array}\right)\ .
$$
Condition (i) is easy to check and condition (ii) is trivially verified because $\omega$ is constant. To check condition (iii) we note that, for the vector $\tau=\cos\alpha\,{\bf e}_1+\sin\alpha\,{\bf e}_2$, we have
$$
\langle\omega;\tau,\cdot\rangle=\left(\begin{array}{rr}
\frac{1}{2}\cos\alpha+\frac{\sqrt{3}}{2}\sin\alpha\\
\frac{1}{2}\cos\alpha-\frac{\sqrt{3}}{2}\sin\alpha
\end{array}\right)\ .
$$
In order to compute the comass norm of $\omega$, we could use the characterization of the norm $\|\cdot\|_{E^*}$ given in Remark \ref{normestar}, but for $n=3$ computations are simpler. 
Since the unit ball of $E$ is convex, and its extreme points are the unit points of $G$, then it is sufficient to evaluate $\langle\omega;\tau,\cdot\rangle$ on $\pm g_1, \pm g_2, \pm (g_1+g_2)$. We have
\begin{align*}
& |\langle\omega;\tau,g_1\rangle|=|\langle\omega;\tau,-g_1\rangle|=\left|\sin\left(\alpha+\frac{\pi}{6}\right)\right|\leq 1\ ,\\
& |\langle\omega;\tau,g_2\rangle|=|\langle\omega;\tau,-g_2\rangle|=\left|\sin\left(\alpha+\frac{5}{6}\pi\right)\right|\leq 1\ ,\\
& |\langle\omega;\tau,g_1+g_2\rangle|=|\langle\omega;\tau,-(g_1+g_2)\rangle|=|\cos\alpha|\leq 1\ .
\end{align*}

\begin{figure}[htbp]
\begin{center}
\scalebox{1}{
\input{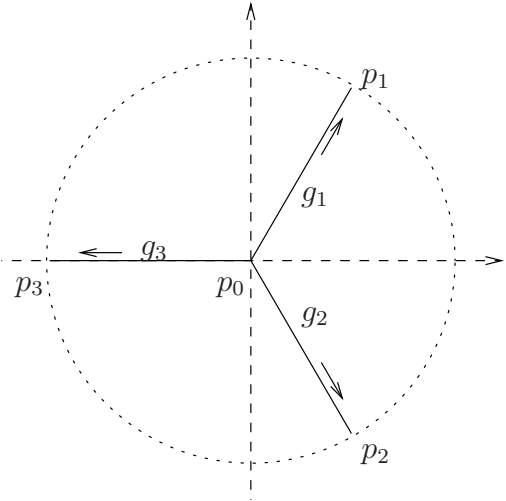}
}
\end{center}
\caption{Solution for the problem with boundary on the vertices of an equilateral triangle}
\label{fig:triang_sol}
\end{figure}

}
\end{example}

In Definition \ref{def_calib} we intentionally kept vague the regularity of the form $\omega$. Indeed $\omega$ has to be a compactly supported\footnote{Since we deal with currents that are compactly supported,
we can easily drop the assumption that $\omega$ has compact support.} smooth form, a priori, in order to fit Definition \ref{def_current}. Nevertheless, in some situations it will be useful to consider
calibrations with lower regularity, for instance piecewise constant forms. As long as \eqref{calib_line}-\eqref{calib_line4} remain valid, it is meaningful to do so; for this reason we introduce the following
very general definition.
\begin{definition}\label{def:gen_calib}{\rm
A \emph{generalized calibration}\index{generalized calibration} associated with a $k$-dimensional normal $E$-current $T$ is a linear and bounded functional $\phi$ on the space of normal $E$-currents satisfying
the following conditions:
\begin{enumerate}
\item[(i)] $\phi(T)=\mathds{M}(T)$;
\item[(ii)] $\phi(\partial R)=0$ for any $(k+1)$-dimensional normal $E$-current $R$;
\item[(iii)] $\|\phi\|\le 1$.
\end{enumerate}
}
\end{definition}
\begin{remark}{\rm
Proposition \ref{prop:calib} still holds, since for every competitor $T'$ with $\partial T=\partial T'$, there holds
$$
\mathds{M}(T)=\phi(T)=\phi(T')+\phi(\partial R)\le\mathds{M}(T')\ ,
$$
where $R$ is chosen such that $T-T'=\partial R$. Such $R$ exists because $T$ and $T'$ are in the same homology class.
}
\end{remark}


As examples, we present the calibrations for two well-known Steiner tree problems in $\R^2$. Both ``calibrations'' in Example \ref{ex:calib_square} and in Example \ref{ex:calib_hex_vittone} are piecewise
constant $1$-forms (with values in normed vector spaces of dimension $3$ and $6$, respectively). So firstly we need to show that certain piecewise constant forms provide generalized calibrations in the sense
of Definition \ref{def:gen_calib}.

\begin{definition}\label{def_compatib_calibration}{\rm
Fix a 1-dimensional rectifiable $G$-current $T$ in $\R^2$, $T=T(\Sigma,\tau,\theta)$. Assume we have a collection $\{C_r\}_{r\ge 1}$ which is a locally finite, Lipschitz partition of $\R^2$, where the sets $C_r$
have non empty connected interior, the boundary of every set $C_r$ is a Lipschitz curve (of finite length, unless $C_r$ is unbounded) and $C_{r}\cap C_{s}=\emptyset$ whenever $r\neq s$. Assume moreover that $C_1$
is a closed set and for every $r>1$
$$C_r\supset (\overline{C_r}\setminus\bigcup_{i<r} C_i).$$
Let us consider a compactly supported piecewise constant $E^*$-valued $1$-form $\omega$ with
$$
\omega\equiv\omega_r\;\;\; {\rm{on}}\;C_{r}\,
$$
where $\omega_r\in\Lambda^1_E(\R^2)$ for every $r$. In particular $\omega\neq 0$ only on finitely many elements of the partition.
Then we say that $\omega$ represents a \emph{compatible calibration}\index{compatible calibration} for $T$ if the following conditions hold:
\begin{enumerate}
\item[(i)] for $\Haus{1}$-almost every point $x\in \Sigma,\; \langle\omega(x);\tau(x),\theta(x)\rangle=\|\theta(x)\|_E$;
\item[(ii)] for $\Haus{1}$-almost every point $x\in\partial C_{r}\cap \partial C_{s}$ we have
$$
\langle\omega_{r}-\omega_{s};\tau(x),\cdot\rangle=0,
$$
where $\tau$ is tangent to $\partial C_r$;
\item[(iii)] $\|\omega_r\|\le 1$ for every $r$.
\end{enumerate}
We will refer to condition (ii) with the expression of \emph{compatibility condition}\index{compatibility condition} for a piecewise constant form.
}
\end{definition}

\begin{proposition}\label{prop_compatib_calibrat}
Let $\omega$ be a compatible calibration for the rectifiable $G$-current $T$. Then $T$ minimizes the mass among the normal $E$-currents with boundary $\partial T$.
\end{proposition}

To prove this proposition we need the following result of decomposition of classical normal $1$-currents, see \cite{smirnov} for the classical result and \cite{paolini} for its generalization to metric spaces.
Given a compact measure space $(X,\mu)$ and a family of $k$-currents $\{T_x\}_{x\in X}$ in $\R^d$, such that
$$
\int_X\mathds{M}(T_x)\,{\rm d}\mu(x)<+\infty\,,
$$
we denote by
$$
T:=\int_X T_x\,{\rm d}\mu(x)
$$
the $k$-current $T$ satisfying
$$
\langle T,\omega\rangle=\int_X\langle T_x,\omega\rangle\,{\rm d}\mu(x)\,,
$$
for every smooth compactly supported $k$-form $\omega$.

\begin{proposition}\label{normaldecomposition}
Every normal 1-current $T$ in $\R^d$ can be written as
$$
T=\int_0^M T_t\; {\rm d}t,
$$
where $T_t$ is an integral current with $\mathds{M}(T_t)\le 2$ and $\mathds{M}(\partial T_t)\le 2$ for every $t$, and $M$ is a positive number depending only on $\mathds{M}(T)$ and $\mathds{M}(\partial T)$. Moreover
$$\mathds{M}(T)=\int_0^M\mathds{M}(T_t)\;{\rm d}t\,.$$
\end{proposition}

\begin{proof}[Proof of Proposition \ref{prop_compatib_calibrat}]
Firstly we see that a suitable counterpart of Stokes Theorem holds. Namely, given a component $\omega^j$ of $\omega$ and a classical integral 1-current $T=T(\Sigma,\tau,1)$ in $\R^2$, without boundary, then we claim that
\begin{equation}\label{stok}
 \langle\omega^j;T\rangle:=\int_{\Sigma}\langle\omega^j(x);\tau(x)\rangle {\rm{d}}\Haus{1}(x)=0.
\end{equation}

To prove such claim, note that it is possible to find at most countably many unit multiplicity integral 1-currents $T_i=T(\Sigma_i,\tau_i,1)$ in $\R^2$, without boundary, each one supported in a single set
$\overline{C_r}$, such that $\sum_i T_i=T$. Since $\omega^j\equiv\omega^j_r$ on $C_r$ and since (ii) holds, then
$$\int_{\Sigma_i}\langle\omega^j(x);\tau_i(x)\rangle {\rm{d}}\Haus{1}(x)=\int_{\Sigma_i}\langle\omega^j_r(x);\tau_i(x)\rangle {\rm{d}}\Haus{1}(x)=0$$
for every $i$, then the claim follows.

As a consequence of (\ref{stok}) we can find family of ``potentials'', i.e. Lipschitz functions $\phi_j:\R^2\rightarrow\R$ such that for every (classical) integral 1-current $S$ associated to
a Lipschitz path $\gamma$ with $\gamma(1)=x_S$ and $\gamma(0)=y_S$, there holds:
$$\langle\omega^j;S\rangle=\phi_j(x_S)-\phi_j(y_S), \;\;\;{\rm{for\;every}}\;j.$$
Indeed, by (\ref{stok}) the above integral does not depend on the path $\gamma$ but only on the points $x_S$ and $y_S$.
Therefore, in order to construct such potentials, it is sufficient to choose $\phi_j(0)=0$ and
$$\phi_j(x)=|x|\int_0^1\langle\omega^j(tx);\frac{x}{|x|}\rangle\;{\rm{d}}t.$$
Moreover it is easy to see that every $\phi_j$ is constant outside of the support of $\omega^j$, so we can assume, possibly subtracting a constant, that $\phi_j$ is compactly supported.

Now, consider any $2$-dimensional normal $E$-current $T$. Let $\{T^j\}_j$ be the components of $T$. For every $j$, use Proposition \ref{normaldecomposition} to write $S^j:=\partial T^j=\int_0^{M_j} S^j_t\,{\rm{d}}t$. Then we have
$$
\langle\omega;\partial T\rangle=\sum_j\int_0^{M_j}\langle \omega^j;S^j_t\rangle\;{\rm{d}}t=\sum_j\int_0^{M_j}\phi_j(x_{S^j_t})-\phi_j(y_{S^j_t})\;{\rm{d}}t.
$$
Since for every $j$ we have
$$
0=\partial(\partial T^j)=\int_0^{M_j}\delta_{x_{S^j_t}}-\delta_{y_{S^j_t}}\;{\rm{d}}t,
$$
then we must have
$$
\int_0^{M_j} g(x_{S^j_t})-g(y_{S^j_t})\;{\rm{d}}t=0,
$$
for every $j$ and for every compactly supported Lipschitz function $g$, in particular for $g=\phi_j$. Hence we have $\langle\omega;\partial T\rangle=0$.
\end{proof}

\begin{example}\label{ex:calib_square}{\rm
Consider the points
$$p_1=(1,1),p_2=(1,-1),p_3=(-1,-1),p_4=(-1,1)\in\R^2.$$
The corresponding solution of the Steiner tree problem\footnote{In dimension $d>2$, an interesting question related to this problem is the following: is the cone over the $(d-2)$-skeleton of the hypercube in $\R^d$
area minimizing, among hypersurfaces separating the faces? The question has a positive answer if and only if $d\ge 4$ (see \cite{brakke1} for the proof).
} are those represented in Figure \ref{fig:sol_steiner}. We associate with each point $p_j$ with $j=1,\ldots,4$ the coefficients $g_j\in G$, where $G$ is the group defined in Lemma \ref{representation} with $n=4$:
let us call
$$
B:=g_1\delta_{p_1}+g_2\delta_{p_2}+g_3\delta_{p_3}+g_4\delta_{p_4}\ .
$$
This $0$-dimensional current is our boundary. Intuitively our mass-minimizing candidates among $1$-dimensional rectifiable $G$-currents are those represented in Figure \ref{fig:sol_bigex}: these currents
$T_{\rm hor},T_{\rm ver}$ are supported in the sets drawn, respectively, with continuous and dashed lines in Figure \ref{fig:sol_bigex} and have piecewise constant coefficients intended to satisfy the boundary
condition $\partial T_{\rm hor}=B=\partial T_{\rm ver}$.

\begin{figure}[htbp]
\begin{center}
\scalebox{1}{
\input{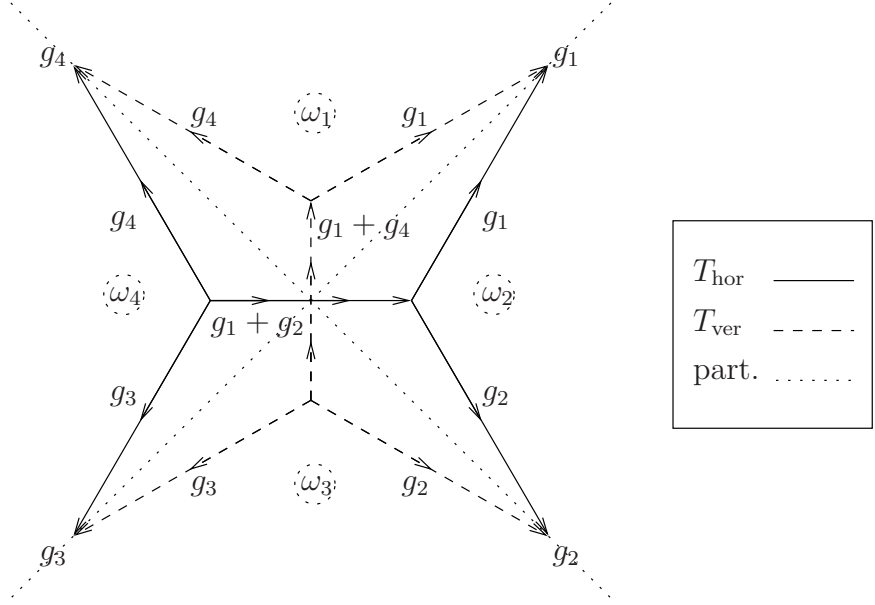}
}
\caption{Solution for the mass minimization problem}
\label{fig:sol_bigex}
\end{center}
\end{figure}

In this case, a compatible calibration for both $T_{\rm hor}$ and $T_{\rm ver}$ is defined piecewise as follows (the notation is the same as in Example \ref{ex:triangle} and the partition is delimited by the dotted
lines):
\begin{eqnarray*}
& \omega_1\equiv\left(
\begin{array}{rcr}
\frac{\sqrt{3}}{2} {\rm{d}}x_1   &\!\!\! + \!\!\!& \frac{1}{2} {\rm{d}}x_2\\
\left(1-\frac{\sqrt{3}}{2}\right) {\rm{d}}x_1 &\!\!\! - \!\!\!& \frac{1}{2} {\rm{d}}x_2\\
\left(-1+\frac{\sqrt{3}}{2}\right) {\rm{d}}x_1 &\!\!\! - \!\!\!& \frac{1}{2} {\rm{d}}x_2
\end{array}
\right)
\quad\ &
\omega_2\equiv\left(
\begin{array}{rcr}
 \frac{1}{2}{\rm{d}}x_1 &\!\!\! + \!\!\!&  \frac{\sqrt{3}}{2}{\rm{d}}x_2  \\
 \frac{1}{2}{\rm{d}}x_1 &\!\!\! - \!\!\!& \frac{\sqrt{3}}{2} {\rm{d}}x_2 \\
-\frac{1}{2}{\rm{d}}x_1 &\!\!\! - \!\!\!& \left(1-\frac{\sqrt{3}}{2}\right){\rm{d}}x_2
\end{array}
\right)\\
& \omega_3\equiv\left(
\begin{array}{rcr}
\left(1-\frac{\sqrt{3}}{2}\right){\rm{d}}x_1&\!\!\! + \!\!\!&  \frac{1}{2}{\rm{d}}x_2\\
\frac{\sqrt{3}}{2}{\rm{d}}x_1 &\!\!\! - \!\!\!   &\frac{1}{2}{\rm{d}}x_2\\
-\frac{\sqrt{3}}{2}{\rm{d}}x_1  &\!\!\! - \!\!\!& \frac{1}{2}{\rm{d}}x_2
\end{array}
\right)
\quad\ &
\omega_4\equiv\left(
\begin{array}{rcr}
 \frac{1}{2}{\rm{d}}x_1&\!\!\! + \!\!\! & \left(1-\frac{\sqrt{3}}{2}\right){\rm{d}}x_2\\
 \frac{1}{2}{\rm{d}}x_1&\!\!\! - \!\!\! & \left(1-\frac{\sqrt{3}}{2}\right){\rm{d}}x_2\\
-\frac{1}{2}{\rm{d}}x_1 &\!\!\! - \!\!\!& \frac{\sqrt{3}}{2}{\rm{d}}x_2
\end{array}
\right)
\end{eqnarray*}

\noindent It is easy to check that $\omega$ satisfies both condition (i) and the compatibility condition of Definition \ref{def_compatib_calibration}. To check that condition (iii) is satisfied, we can use formula
\eqref{char_norm_Estar}.
}
\end{example}

\begin{example}\label{ex:calib_hex_vittone}{\rm
Consider the vertices of a regular hexagon plus the center, namely
\begin{eqnarray*}
p_1=(1/2,\sqrt{3}/2),\   & p_2=(1,0),\  & p_3=(1/2,-\sqrt{3}/2),\\
p_4=(-1/2,-\sqrt{3}/2),\ & p_5=(-1,0),\ & p_6=(-1/2,\sqrt{3}/2),\quad p_7=(0,0)
\end{eqnarray*}
and associate with each point $p_j$ the corresponding multiplicity $g_j\in G$, where $G$ is the group defined in Lemma \ref{representation} with $n=4$. A mass-minimizer for the problem with boundary
$$
B=\sum_{j=1}^7 g_j\delta_{p_j}
$$
is illustrated in Figure \ref{fig:sol_hex}, the other one can be obtained with a $\pi/3$-rotation of the picture.

\begin{figure}[htbp]
\begin{center}
\scalebox{1}{
\input{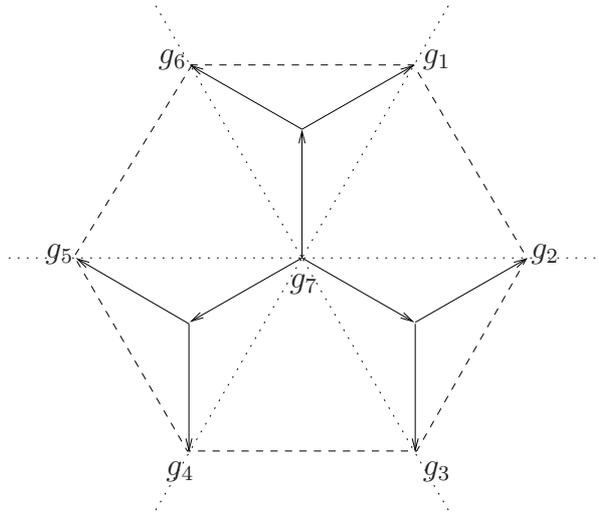}
}
\caption{Solution for the mass minimization problem}
\label{fig:sol_hex}
\end{center}
\end{figure}

Let us divide $\R^2$ in $6$ cones of angle $\pi/3$, as in Figure \ref{fig:sol_hex}; we will label each cone with a number from $1$ to $6$, starting from that containing $(0,1)$ and moving clockwise.
A compatible calibration for the two minimizers is the following

\begin{equation}\label{calib_hex_vittone}
\begin{array}{lll}
\omega_1=\left(\!\!\!
\begin{array}{rr}
-\frac{\sqrt{3}}{2}{\rm{d}}x_1 \!\!\!&\!\!\! + \frac{1}{2}{\rm{d}}x_2 \\
 \frac{\sqrt{3}}{2}{\rm{d}}x_1 \!\!\!&\!\!\! + \frac{1}{2}{\rm{d}}x_2 \\
                   &0\hphantom{\frac{1}{2}{\rm{d}}x_2}   \\
                   &0\hphantom{\frac{1}{2}{\rm{d}}x_2}   \\
                   &0 \hphantom{\frac{1}{2}{\rm{d}}x_2}             \\
                   &0 \hphantom{\frac{1}{2}{\rm{d}}x_2}             \\
\end{array}
\!\!\!\right)\ & \
\omega_2=\left(\!\!\!
\begin{array}{rr}
                  & 0\hphantom{\frac{1}{2}{\rm{d}}x_2}          \\
                  &  {\rm{d}}x_2\hphantom{\frac{1}{2}0}          \\
 \frac{\sqrt{3}}{2}{\rm{d}}x_1\!\!\! & \!\!\!-\frac{1}{2}{\rm{d}}x_2 \\
                  & 0\hphantom{\frac{1}{2}{\rm{d}}x_2}					 \\
                  & 0\hphantom{\frac{1}{2}{\rm{d}}x_2}           \\
                  & 0\hphantom{\frac{1}{2}{\rm{d}}x_2}           \\
\end{array}
\!\!\!\right)\ &\
\omega_3=\left(\!\!\!
\begin{array}{rr}
                 & 0\hphantom{\frac{1}{2}{\rm{d}}x_2} \\
									&	0\hphantom{\frac{1}{2}{\rm{d}}x_2} \\
 \frac{\sqrt{3}}{2}{\rm{d}}x_1\!\!\! &\!\!\!+\frac{1}{2}{\rm{d}}x_2 \\
                  & \!\!\!-{\rm{d}}x_2\hphantom{\frac{1}{2}} \\
                  & 0\hphantom{\frac{1}{2}{\rm{d}}x_2} \\
                  & 0\hphantom{\frac{1}{2}{\rm{d}}x_2} \\
\end{array}
\!\!\!\right)\\
\omega_4=\left(\!\!\!
\begin{array}{rr}
                  & 0\hphantom{\frac{1}{2}{\rm{d}}x_2}           \\
                  & 0\hphantom{\frac{1}{2}{\rm{d}}x_2}           \\
                  & 0\hphantom{\frac{1}{2}{\rm{d}}x_2}           \\
 \frac{\sqrt{3}}{2}{\rm{d}}x_1\!\!\! & \!\!\!-\frac{1}{2}{\rm{d}}x_2 \\
-\frac{\sqrt{3}}{2}{\rm{d}}x_1\!\!\! & \!\!\!-\frac{1}{2}{\rm{d}}x_2 \\
                  & 0\hphantom{\frac{1}{2}{\rm{d}}x_2}           \\
\end{array}
\!\!\!\right)\ & \
\omega_5=\left(\!\!\!
\begin{array}{rr}
                  & 0\hphantom{\frac{1}{2}{\rm{d}}x_2}           \\
                  & 0\hphantom{\frac{1}{2}{\rm{d}}x_2}           \\
                  & 0\hphantom{\frac{1}{2}{\rm{d}}x_2}           \\
                  & 0\hphantom{\frac{1}{2}{\rm{d}}x_2}           \\
                  & \!\!\!-{\rm{d}}x_2\hphantom{\frac{1}{2}}          \\
-\frac{\sqrt{3}}{2}{\rm{d}}x_1\!\!\! &\!\!\! +\frac{1}{2}{\rm{d}}x_2 \\
\end{array}
\!\!\!\right)\ & \
\omega_6=\left(\!\!\!
\begin{array}{rr}
                 &  {\rm{d}}x_2\hphantom{\frac{1}{2}0}          \\
                  &  0\hphantom{\frac{1}{2}{\rm{d}}x_2}           \\
                  &  0\hphantom{\frac{1}{2}{\rm{d}}x_2}           \\
                  &  0\hphantom{\frac{1}{2}{\rm{d}}x_2}           \\
                  &  0\hphantom{\frac{1}{2}{\rm{d}}x_2}           \\
-\frac{\sqrt{3}}{2}{\rm{d}}x_1\!\!\! & \!\!\!-\frac{1}{2}{\rm{d}}x_2 \\
\end{array}
\!\!\!\right)
\end{array}
\end{equation}

Again, it is not difficult to check that $\omega$ satisfies both condition (i) and the compatibility condition of Definition \ref{def_compatib_calibration}. To check that condition (iii) is satisfied, we use
formula \eqref{char_norm_Estar}.
}
\end{example}

\begin{remark}{\rm
We may wonder whether or not the calibration given in Example \ref{ex:calib_hex_vittone} can be adjusted so to work for the set of the vertices of the hexagon (without the seventh point in the center): the answer
is negative, in fact the support of the current in Figure \ref{fig:sol_hex} is not a solution for the Steiner tree problem on the six points, the perimeter of the hexagon minus one side being the shortest graph,
as proved in \cite{jarnik}.
}
\end{remark}

\begin{remark}{\rm
In both Examples \ref{ex:calib_square} and \ref{ex:calib_hex_vittone}, once we fixed the partition and we decided to look for a piecewise constant calibration for our candidates, the construction of $\omega$ was
forced by both conditions (i) of Definition \ref{def_calib} and the compatibility condition of Definition \ref{def_compatib_calibration}. Notice that the calibration for the Example \ref{ex:calib_hex_vittone}
has evident analogies with the one exhibited in the Example \ref{ex:triangle}. Actually we obtained the first one simply pasting suitably ``rotated'' copies of the second one.
}
\end{remark}

In the following remarks we intend to underline the analogies and the connections with calibrations in similar contexts. See Chapter 6 of \cite{morgan2} for an overview on the subject of calibrations.

\begin{remark}[Functionals defined on partitions and null lagrangians]\label{rmk:null_lagrangian}{\rm
There is an interesting and deep analogy between calibrations and null lagrangians, analogy that still holds in the group-valued coefficients framework. Consider some points $\{\eta_1,\ldots,\eta_n\}\subset\R^m$,
with
\begin{equation}\label{rmk:nl_cond_coeff}
|\eta_i-\eta_j|=1\quad\forall\,i\neq j\ ,
\end{equation}
and fix an open set with Lipschitz boundary $\Omega\subset\R^d$.
It is natural to study the variational problem
\begin{equation}\label{rmk:var_pb_nl}
\inf\left\{\int_\Omega|Du|:\,u\in BV\left(\Omega;\{\eta_1,\ldots,\eta_n\}\right),u_{|\partial\Omega}\equiv u_0\right\}\ .
\end{equation}
It turns out that $\int_\Omega|Du|$ is the same energy we want to minimize in the Steiner tree problem, $\int_\Omega|Du|$ being the length of the jump set of $u$.

\begin{figure}[htbp]
\begin{center}
\scalebox{1}{
\input{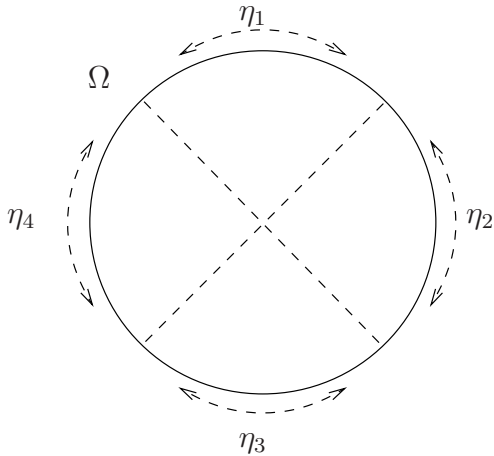}
}
\end{center}
\caption{Boundary data}
\label{fig:nl_boundary}
\end{figure}

This problem concerns the theory of partitions of an open set $\Omega$ in a finite number of sets of finite perimeter. This theory was developed by Ambrosio and Braides in \cite{Am_Bra1,Am_Bra2}, which we
refer to for a complete exposition.

The analog of a calibration in this context is a null lagrangian\footnote{See \cite{Da} for an overview on null lagragians.} with some special properties: again, the existence of such an object, associated with a function $u$, is a sufficient condition for $u$ to be a minimizer for the variational problem \eqref{rmk:var_pb_nl} with a given boundary condition.

We refer to $3.2.4$ in \cite{massaccesi} for a detailed survey of the analogy.
}
\end{remark}


\begin{remark}[Clusters with multiplicities]\label{rmk:morgan_calib}{\rm
In \cite{morgan}, F. Morgan applies flat chains with coefficients in a group $G$ to soap bubble clusters and immiscible fluids, following the idea of B. White in \cite{white0}. For a detailed comparison of \cite{morgan} with our technique, see $3.2.3$ of \cite{massaccesi}. Here we just notice that the definition of calibration in \cite{morgan} works well in the case of free abelian groups and this is the main difference with our approach.
}
\end{remark}

\begin{remark}[Paired calibrations]\label{rmk:paired_calib}{\rm
It is worth mentioning another analogy between the technique of calibrations (for currents with coefficients in a group) illustrated in this paper and the technique of paired calibrations in \cite{Law_Mor}. In particular, in the specific example of the truncated cone over the $1$-skeleton of the tetrahedron in $\R^3$ (the surface with least area among
those separating the faces of the tetrahedron), one can detect a correspondence even at the level of the main computations. See $3.2.3$ of \cite{massaccesi} for the details.
}
\end{remark}


Following an idea of Federer (see \cite{federer}), in \cite{morgan} and \cite{Law_Mor} (and in \cite{brakke1} and \cite{brakke2}, as well) one can observe the exploitation of the duality between minimal
surfaces and maximal flows through the same boundary. We will examine this duality in \S \ref{sec:open_pb}, but we conclude the present section with a remark closely related to this idea.

\begin{remark}[Covering spaces and calibrations for soap films]{\rm
In \cite{brakke2} Brakke develops new tools in Geometric Measure Theory for the analysis of soap films: as the underlying physical problem suggests, one can represent a soap film as the superposition of two
oppositely oriented currents. In order to avoid cancellations of multiplicities, the currents are defined in a covering space and, as stated in \cite{brakke2}, the calibration technique still holds.

Let us remark that cancellations between multiplicities were a significant obstacle for the Steiner tree problem, too. The representation of currents in a covering space goes in the same direction of currents
with coefficients in a group, though, as in Remark \ref{rmk:paired_calib}, a sort of Poincar\'e duality occurs in the formulation of the Steiner tree problem ($1$-dimensional currents in $\R^d$) with respect
to the soap film problem (currents of codimension $1$ in $\R^d$).
}
\end{remark}


\section{Existence of the calibration and open problems}\label{sec:open_pb}

Once we established that the existence of a calibration is a sufficient condition for a rectifiable $G$-current to be a mass-minimizer, we may wonder if the converse is also true: does a calibration (of some sort)
exist for every mass-minimizing rectifiable $G$-current?

Let us step backward: does it occur for classical integral currents? The answer is quite articulate, but we can briefly summarize the state of the art we will rely upon. 

We consider a boundary $B_0$, that is, a $(k-1)$-dimensional rectifiable $G$-current without boundary, and we compare the following minima:
$$
{\mathscr M}_E(B_0):=\min\{\mathds{M}(T):\,T\ {\rm is\;a\;normal}\;k-{\rm dimensional}\;E{\rm -current},\,\partial T=B_0\}
$$
and
$$
{\mathscr {M}}_G(B_0):=\min\{\mathds{M}(T):\,T\ {\rm is\;a\;rectifiable}\;k-{\rm dimensional}\;G{\rm -current},\,\partial T=B_0\}.
$$
Obviously ${\mathscr{M}}_E(B_0)\le{\mathscr{M}}_G(B_0)$, the main issue is to establish whether they coincide or not. In fact, a normal $E$-current $T$ with boundary $B_0$ admits a generalized calibration if and only if $\mathds{M}(T)={\mathscr M}_E(B_0)$, as we recall in Proposition \ref{E-fed}. In the classical case ($E=\R$ and $G=\Z$) it is known that
\begin{enumerate}
\item[(i)] ${\mathscr{M}}_\R(B_0)$ may be strictly less than ${\mathscr {M}}_\Z(B_0)$ (and, if this happens, a solution for  ${\mathscr{M}}_\Z(B_0)$ cannot be calibrated);\item[(ii)]
${\mathscr {M}}_\Z(B_0)={\mathscr {M}}_\R(B_0)$ if $k=1$, as we prove in Proposition \ref{thm_minmin_classic}.
\end{enumerate}

At the end of this section, we show that this outlook changes significantly when we replace the ambient space $\R^d$ with a suitable metric space.

\begin{remark}\label{rmk_fed_arg}{\rm
For every mass-minimizing classical normal $k$-current $T$, there exists a generalized calibration $\phi$ in the sense of Definition \ref{def:gen_calib}. Moreover, by means of the Riesz Representation Theorem,
$\phi$ can be represented by a measurable map from $U$ to $\Lambda^k(\R^d)$.
This result is contained in \cite{federer}.
}
\end{remark}

In particular, Remark \ref{rmk_fed_arg} provides a positive answer to the question of the existence of a generalized calibration for mass-minimizing integral currents of dimension $k=1$, because minima
among both normal and integral currents coincide, as we prove in Proposition \ref{thm_minmin_classic}. It is possible to apply the same technique in the class of normal $E$-currents, therefore we have the
following proposition.

\begin{proposition}\label{E-fed}
For every mass minimizing normal $E$-current $T$, there exists a ge\-ne\-ra\-li\-zed calibration.
\end{proposition}

The following fact is probably in the folklore, unfortunately we were not able to find any literature on it. We give a proof here in order to enlighten the problems arising in the case of currents with coefficients in a group.

\begin{proposition}\label{thm_minmin_classic}
Consider the boundary of an integral $1$-current in $\R^d$, represented as
\begin{equation}\label{eq_boundary_minmin}
B_0=-\sum_{i=1}^{N_-} a_i\delta_{x_i}+\sum_{j=1}^{N_+} b_j\delta_{y_j},\quad a_i,b_j\in\N\ .
\end{equation}
Then ${\mathscr M}_\R(B_0)={\mathscr M}_\Z(B_0).$
\end{proposition}

\begin{Proof}
Let us assume that the minimum among normal currents is attained at some current $T_0$, that is
$$
\mathds{M}(T_0)={\mathscr M}_\R(B_0)\ .
$$

Let $\{T_h\}_{h\in\N}$ be an approximation of $T_0$ made by polyhedral $1$-currents, such that
\begin{itemize}
\item $\mathds{M}(T_h)\rightarrow \mathds{M}(T_0)$ as $h\to\infty$,
\item $\partial T_h=B_0$ for all $h\in\N$,
\item the multiplicities allowed in $T_h$ are only integer multiples of $\frac{1}{h}$.
\end{itemize}

The existence of such a sequence is a consequence of the Polyhedral Approximation Theorem (see Theorem 4.2.24 of \cite{Fed} or \cite{krantz} for the detailed statement and the proof). Thanks
to Theorem \ref{thm:struct_1currents}, it is possible to decompose such a $T_h$ as a sum of two addenda:
\begin{equation}\label{dec_seq}
T_h=P_h+C_h\ ,
\end{equation}
so that
$$
\mathds{M}(T_h)=\mathds{M}(P_h)+\mathds{M}(C_h)\quad\forall\,h\ge 1
$$
and
\begin{itemize}
\item $\partial C_h=0$, so $C_h$ collects the cyclical part of $T_h$;
\item $P_h$ does not admit any decomposition $P_h=A+B$ satisfying $\partial A=0$ and $\mathds{M}(P_h)=\mathds{M}(A)+\mathds{M}(B)$
\end{itemize}

It is clear that $P_h$ is the sum of a certain number of polyhedral currents $P_h^{i,j}$ each one having boundary a non-negative multiple of $-\frac{1}{h}\delta_{x_i}+\frac{1}{h}\delta_{y_j}$ and satisfying
$$\mathds{M}(P_h)=\sum_{i,j}\mathds{M}(P_h^{i,j})$$
We replace each $P_h^{i,j}$ with the oriented segment $Q^{i,j}$, from $x_i$ to $y_j$
having the same boundary as $P_h^{i,j}$ (therefore having multiplicity a non-negative multiple of $\frac{1}{h}$). This replacement is represented in Figure \ref{fig:poly}
\begin{figure}[htbp]
\begin{center}
\scalebox{1}{
\input{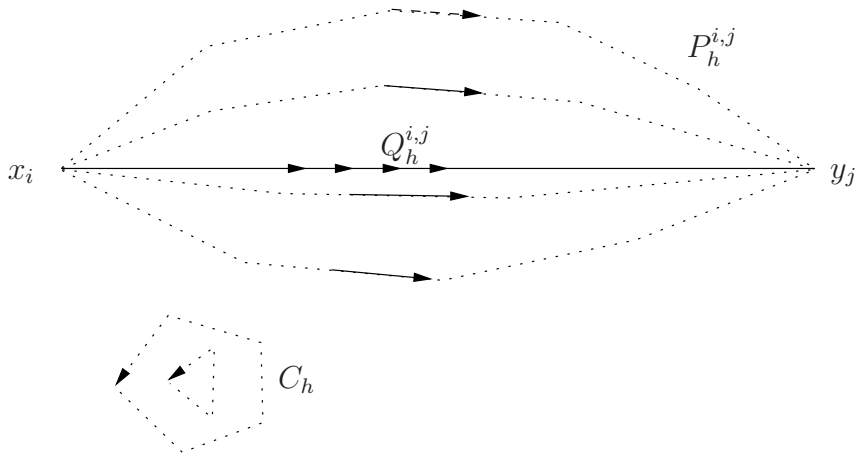}
}
\end{center}
\caption{Replacement with a segment}
\label{fig:poly}
\end{figure}

Since this replacement obviously does not increase the mass, there holds $\mathds{M}(P_h)\geq\mathds{M}(Q_h),$ where $Q_h=\sum_{i,j}Q_h^{i,j}$. In other words we can write
$Q_h=\int_{I}T\,{\rm{d}}\lambda_h,$
as an integral of currents, with respect to a discrete measure $\lambda_h$ supported on the finite set $I$ of unit multiplicity oriented segments with the first extreme among the points $x_1,\ldots,x_{N_-}$
and second extreme among the points $y_1,\ldots,y_{N_+}$. It is also easy to see that the total variation of $\lambda_h$ has eventually the following bound from above
$$
\|\lambda_h\|\le\frac{\mathds{M}(T_h)}{\min_{i\neq j} d(x_i,y_j)}\le\frac{\mathds{M}(T_0)+1}{\min_{i\neq j} d(x_i,y_j)}\ .
$$

\noindent Hence, up to subsequences, $\lambda_h$ converges to some positive measure $\lambda$ on $I$ and so the normal $1$-current
$$
Q=\int_{T\in I}T\,{\rm{d}}\lambda
$$
satisfies
\begin{equation}\label{eq_q_bound}
\partial Q =B_0
\end{equation}
and
$$
\mathds{M}(Q)\le \mathds{M}(T_0)={\mathscr M}_N(B_0)\ .
$$

\noindent In order to conclude the proof of the theorem, we need to show that $Q$ can be replaced by an integral current $R$ with same boundary and mass $\mathds{M}(R)=\mathds{M}(Q)\le{\mathscr M}_N(B_0)$.
Since $I$ is the set of unit multiplicity oriented segments $\Sigma^{ij}$ from $x_i$ to $y_j$, we can obviously represent
$$
Q=\sum_{i,j} k^{ij}\Sigma^{ij}\quad{\rm with}\ k^{ij}\in\R\ ,
$$
and, again, thanks to \eqref{eq_q_bound},
$$
\sum_{i=1}^{N_-}k^{ij}=b_j\quad{\rm and}\quad\sum_{j=1}^{N_+}k^{ij}=a_i\ .
$$

\noindent If $k^{ij}\in\Z$ for any $i,j$, then $Q$ itself is integral and then we are done; if not, let us consider the finite set of non-integer multiplicities
$$
K_{\R\setminus\Z}:=\left\{k^{ij}\,:\,i=1,\ldots,N_-,\,j=1,\ldots,N_+\right\}\setminus\Z\neq\emptyset\ .
$$

\noindent We fix $k\in K_{\R\setminus\Z}$ and we choose an index $(i_0,j_0)$, such that $k$ is the multiplicity of the oriented segment $\Sigma^{i_0 j_0}$ in $Q$. It is possible to track down a non-trivial
cycle $\overline Q$ in $Q$ with the following algorithm: after $\Sigma^{i_0 j_0}$, choose a segment from $x_{i_1}\neq x_{i_0}$ to $y_{j_0}$ with non-integer multiplicity, it must exist because $B_0=\partial Q$
is integral. Then choose a segment from $x_{i_1}$ to $y_{j_1}\neq y_{j_0}$ with non-integer multiplicity and so on. Since $K_{\R\setminus\Z}$ is finite, at some moment we will get a cycle. Up to reordering the
indices $i$ and $j$ we can write
$$
\overline Q=\sum_{l=1}^n(\Sigma^{i_l j_l}-\Sigma^{i_{l+1} j_l})\ .
$$

\noindent We will denote by
\begin{eqnarray*}
\alpha &:=& \min_l (k^{i_l j_l}-\lfloor k^{i_l j_l}\rfloor)>0\\
\beta  &:=& \min_l (k^{i_{l+1} j_l}-\lfloor k^{i_{l+1} j_l}\rfloor)>0\ .
\end{eqnarray*}

\noindent Finally notice that both $Q-\alpha\overline Q$ and $Q+\beta\overline Q$ have lost at least one non-integer coefficient; in addition, we claim that either
\begin{equation}\label{concl_minmin}
\mathds{M}(Q-\alpha\overline Q)\le\mathds{M}(Q)\quad{\rm or}\quad\mathds{M}(Q+\beta\overline Q)\le\mathds{M}(Q)\ .
\end{equation}
In fact we can define the linear auxiliary function
$$
F(t):=\mathds{M}(Q)-\mathds{M}(Q-t\overline Q)=\sum_l(k^{i_l j_l}-t)d(x_{i_l},y_{j_l})+(k^{i_{l+1}j_l}+t)d(x_{i_{l+1}},y_{j_l})
$$
for which $F(0)=0$, so either
$$
F(\alpha)\ge 0\quad{\rm or}\quad F(-\beta)\ge 0\ .
$$

Iterating this procedure finitely many times, we obtain an integral current without increasing the mass.
\end{Proof}

In order to guarantee the existence of a generalized calibration also for $1$-dimensio\-nal mass-minimizing rectifiable $G$-currents, we need an analog of Proposition \ref{thm_minmin_classic} in the framework
of $G$-currents. Namely, we need to prove that the minimum of the mass among 1-dimensional normal $E$-currents with the same boundary\footnote{Here
the boundary is of course a $0$-dimensional rectifiable $G$-current.} coincides with the minimum calculated among rectifiable $G$-currents. From the argument used in the proof of Proposition \ref{thm_minmin_classic} we realize that the equality of the two minima in the framework of
 1-dimensional $E$-currents is equivalent to the homogeneity property in Remark \ref{rmk:homog_cond}.

\begin{remark}\label{rmk:homog_cond}{\rm
Fix a $0$-dimensional rectifiable $G$-current $R=\sum_{i=1}^n g_i\delta_{x_i}$ with $\|g_i\|_E=1$ in $U\subset\R^d$, then ${\mathscr M}_E(R)={\mathscr M}_G(R)$ if and only if the following is true: given a mass-minimizing rectifiable $G$-current $T$
with $\partial T=R$, then for every $k\in\N$ we have that
\begin{equation}\label{eq_big_qstmark}
\min\left\{\mathds{M}(S)\,:\,S\ {\rm rectifiable}\ G-{\rm current},\partial S=kR\right\}=k\mathds{M}(T)\ .
\end{equation}
Notice that \eqref{eq_big_qstmark} can be meaningfully rewritten as
\begin{equation}\label{eq_big_short}
{\mathscr M}_G(kR)=k{\mathscr M}_G(R)\ .
\end{equation}
The condition \ref{eq_big_short} is clearly necessary to have the equality of the two minima. It is also sufficient, in fact one can approximate a normal $E$-current with polyhedral currents with coefficients in $\Q G$.
}
\end{remark}

The homogeneity property, which is trivially verified for classical integral currents, seems to be an interesting issue
in the class of rectifiable $G$-currents. In Example \ref{ex_casetta_piccola} we exhibit a subset $M\subset\R^2$ such that, if our currents are forced to be supported on $M$, then the homogeneity property does
not hold. In other words, we can say that equality of the two minima does not hold in the framework of 1-dimensional $E$-currents on the metric space $M$. We can see the same phenomenon if we substitute the metric
space $M$ with the metric space $\R^2$ endowed with a density, which is unitary on the points of $M$ and very high outside.

\begin{figure}[htbp]
\begin{center}
\scalebox{1}{
\input{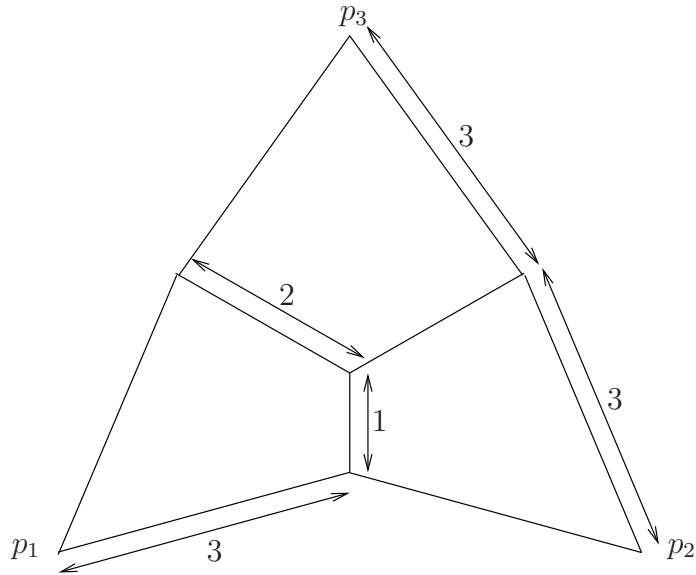}
}
\end{center}
\caption{Metric space in the Example \ref{ex_casetta_piccola}}
\label{fig:casetta}
\end{figure}

\begin{example}\label{ex_casetta_piccola}{\rm
Consider the metric space\footnote{For currents in metric spaces, see \cite{amb_kirch}.} $M\subset\R^2$ given\footnote{The length of each segment is explicitly declared in Figure \ref{fig:casetta}, note that the
set is symmetric with respect to the vertical axis.} in Figure \ref{fig:casetta}. Consider the group $G$, with $n=3$, introduced in \S \ref{sec:steiner} and let
$R:=g_1\delta_{p_1}+g_2\delta_{p_2}+g_3\delta_{p_3}$. We will show that \eqref{eq_big_short} does not hold even when $k=2$. In fact it is trivial to prove that
$$
{\mathscr M}_I(R)=12\ .
$$

\begin{figure}[htbp]
\begin{center}
\scalebox{1}{
\input{pizzaiolo_pazzo.pstex_t}
}
\end{center}
\caption{Counterexample to \eqref{eq_big_short}}
\label{fig:pizzaiolo_pazzo}
\end{figure}

Nevertheless, concerning ${\mathscr M}_I(2R)$, it is shown in Figure \ref{fig:pizzaiolo_pazzo} that
$$
{\mathscr M}_I(2R)\le 23<24=2{\mathscr M}_I(R)\ .
$$
}
\end{example}

\begin{remark}{\rm
One can expect a behavior like that in Example \ref{ex_casetta_piccola} in the metric space $\R^2$ endowed with a density which is very high outside of the subset $M\subset\R^2$. To be precise, let us consider
a bounded continuous function $W:\R^2\to\R$, with $W\equiv 1$ on $M$ and $W>>1$ out of a small neighborhood of $M$. For any couple $(x_0,x_1)\in\R^2$, the distance on $(\R^2,W)$ is given by
$$
d(x_0,x_1)=\inf\left\{\int_0^1|\gamma'(t)|W(\gamma(t))\,{\rm d}t:\,\gamma(0)=x_0\text{ and }\gamma(1)=x_1 \right\}\,.
$$
}
\end{remark}

\newpage
\noindent{\scshape Max-Planck-Institut f\"ur Mathematik in den Naturwissenschaften \\
Inselstrasse 22, 04103 Leipzig, Germany}\\
e-mail: \emph{marchese@mis.mpg.de}\\

\noindent{\scshape Institut f\"ur Mathematik der Universit\"at Z\"urich \\
Winterthurerstrasse 190, CH-8057 Z\"urich, Switzerland}\\
e-mail: \emph{annalisa.massaccesi@math.uzh.ch}

\end{document}